\documentclass{amsart}
\usepackage{amsmath}
\usepackage{mathrsfs}
\usepackage{epsfig}
\usepackage{hyperref}

\usepackage[percent]{overpic}

\newtheorem{theorem}{Theorem}

\newtheorem{notation}{Notation}

\newtheorem{definition}{Definition}
\newtheorem{lemma}{Lemma}
\newcommand{\aplus}{A^+}
\newcommand{\bplus}{B^+}
\newcommand{\aminus}{A^-}
\newcommand{\bminus}{B^-}
\newcommand{\andbill}{\Omega(B^\pm)}
\newcommand{\andflow}[1]{\Phi_\pm^{#1}}
\begin{document}
\title[Polygonal Andreev Billiards]{Properties of the flow on a polygonal Andreev billiard}

\author[R. G. Niemeyer]{Robert G. Niemeyer}
\address{University of Maine, Department of Mathematics \& Statistics.}
\email{niemeyer@math.umaine.edu}
\thanks{The work of R. G. Niemeyer was partially supported by the National Science Foundation under the MCTP grant DMS-1148801, while a postdoctoral fellow at the University of New Mexico, Albuquerque.}
\keywords{fractal billiard, polygonal billiard, rational (polygonal) billiard, law of reflection, Andreev reflection, retro-reflection, billiard flow, iterated function system and attractor, fractal, prefractal approximations, $T$-fractal billiard, prefractal rational billiard approximations, sequence of compatible orbits, (eventually) constant sequences of compatible orbits, elusive points, periodic orbits, periodic vs. dense orbits, topological dichotomy, superconductor, nanowire, perturbation, Cooper pair, hole, electron, charge, Fermi momentum, Fermi energy, spectrum, spectral gap, pseudointegrable, almost integrable}
\subjclass[2010]{Primary: 28A80, 37D40, 37D50, Secondary: 28A75, 37C27, 37E35, 37F40, 58J99.}

\begin{abstract}
A formal definition of a (mathematical) polygonal Andreev billiard $\andbill$ and a construction of an equivalence relation $\sim_\pm$ that captures the dynamics described in physical toy model of Andreev reflection are given.  The continuous flow $\andflow{t}$ and discrete flow $F_\pm^n$ on the respective spaces $(\andbill\times S^1)/\sim_\pm$ and $(B^\pm\times S^1)/\sim_\pm$ are defined.  It is then shown that the continuous flow $\andflow{t}$ preserves the absolute value of the volume element $dx\wedge dy\wedge d\theta$ and the billiard (collision) map preserves the measure $\cos \phi dr d\phi$, respectively.  One can then characterize the dynamics of a rational polygonal Andreev billiard table.  Finally, a discussions of the effect of a fractal perturbation of the toy model of a rectangular nanowire lying upon a superconducting medium is given.
\end{abstract}

\maketitle
\setcounter{tocdepth}{2}
\tableofcontents

\section{Introduction}
A classical billiard, denoted by $\Omega(B)$ where the boundary of $\Omega(B)$ is $B$, is a compact region in the plane. The boundary $B$ is sufficiently smooth so as to allow the classical law of reflection\footnote{The angle of incidence $\theta_i$ is equal to the angle of reflection $\theta_r$.} to hold at all but finitely many points, and a pointmass traversing the billiard does so with unit speed, zero friction and makes perfectly elastic collisions in the boundary $B$ (reflecting according to the law of reflection); see Figure \ref{fig:exampleOfBilliardFlow}.  

\begin{figure}
\centering
\begin{overpic}[scale=.8]{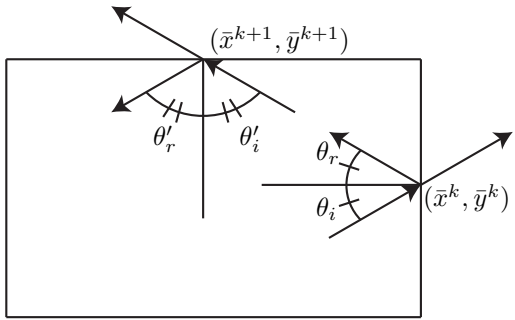}
\put(82,22){$(\bar{x}^k,\bar{y}^k)$}
\put(61,20){$\theta_i$}
\put(61,31){$\theta_r$}
\put(40,53){$(\bar{x}^{k+1},\bar{y}^{k+1})$}
\put(29,34){$\theta_r'$}
\put(46,34){$\theta_i'$}
\end{overpic}
\caption{The standard law of reflection: the angle of incidence is equal to the angle of reflection.  The vector describing the motion of the billiard ball is reflected through the tangent at $(\bar{x}^k,\bar{y}^k)$ and $(\bar{x}^{k+1},\bar{y}^{k+1})$ to produce the new direction of travel $\theta^k$ and $\theta^{k+1}$, respectively,\protect\footnotemark \ at each point.  The collision is perfectly elastic and the billiard ball has unit speed.}
\label{fig:exampleOfBilliardFlow}
\end{figure}

\footnotetext{The directions $\theta^k$ and $\theta^{k+1}$ are not shown in the figure.  Rather, we emphasize that the angle of incidence $\theta_i$ equals the angle of reflection $\theta_r$ at each point.}

The focus of this paper will be primarily on what will be called a \textit{polygonal Andreev billiard} $\andbill:=\Omega(\bplus)\cup\Omega(\bminus)$, with the reflection rule on a subset of $\bplus\cup\bminus$ being nonstandard and having motivation in the physics literature.\footnote{For a detailed introduction to the topic of mathematical billiards, see \cite{Ta1, Ta2,Sm} and Chapters 1--3 in \cite{ChrMar}.  For an excellent introduction to the subject of polygonal billiards---specifically rational billiards---and translation surfaces, see \cite{MasTa}.}  If $\Omega(B)$ is a billiard, then one may construct a well-defined flow $\Phi^t:(\Omega(B)\times S^1)/\sim \to (\Omega(B)\times S^1)/\sim$ that models the flow of a pointmass traversing $\Omega(B)$; the relation $\sim$ identifies inward and outward pointing vectors so that the flow $\Phi^t$ is continuous.\footnote{In truth, the flow is defined everywhere except for at cusps and corners.  We will follow such a convention when defining the flow on an Andreev billiard.}  The flow $\Phi^t$ also preserves the volume element $dx\wedge dy \wedge d\theta$ on $(\Omega(B)\times S^1)/\sim$.  Derived from $\Phi^t$ is the collision map $F:(B\times S^1)/\sim  \to (B\times S^1)/\sim$.  Since $B$ can be parameterized, $(B\times S^1)/\sim$ can be expressed as a two-dimensional space.  The collision map $F$, when defined, preserves the measure 
\[
d\mu = \cos\phi drd\phi
\]
\noindent on $(B\times S^1)/\sim$, where $r$ is the parameterization of $B$ (or, when $B$ is piecewise smooth, $r_i$ parameterizes a smooth subset of $B$ and $B=\cup_{i=1}^k r_i, k$ a positive integer).

The immediate goal of this paper is to rigorously define the notion of Andreev reflection as inspired by the physical phenomenon by the same name, which we now briefly describe; \cite{Bee2} for an excellent survey of the topic of Andreev billiards from the physical perspective; additionally, see \cite{Bee1, Bee2, BerBarBee, GooJacBee, MiPiMacBee, Wi} and the relevant references therein for the state of the art of Andreev reflection in nanowires and other physical settings.  At the nanoscale---this being on the order of magnitude $10^{-9}$---quantum effects become relevant and certain concessions must be made when modeling nanoscale systems.\footnote{Such models are semi-classical in nature.} The specific setting from which this paper draws inspiration is that of a toy model of nanowire (a string of molecules much longer than tall) resting upon a superconducting plate.  As one investigates the behavior of an electron (modeled as a pointmass) in a field of electrons, he or she assumes an electron obeys the law of reflection at the top of the `wire' and then collides in the bottom of the `wire,' and subsequently disappears.  The electron does not truly disappear (nor does it truly strictly obey the law of reflection), but is pulled from the nanowire and forced to pair with another electron, this newly formed pair of electrons being called a \textit{Cooper pair}; see Figure \ref{fig:electronPairing} for an illustration of an idealization of this phenomenon.  It is worth noting that this toy model is valid when the energy in the system is the Fermi energy, this being an important requirement in the study of the solid state physics of metallic substances and superconductors.  Specifically, retro-reflection will not occur when the energy of the system is not the Fermi energy.\footnote{The author would like to thank Carlo Beenakker for pointing out these facts in previous drafts.}

 In the absence of the electron, the `stationary' electrons\footnote{Of course, no electron is ever stationary.  We are indeed imposing  many simplifications on the physical system when deriving a mathematical construct.} react to the departure of the electron by exhibiting a wave (visually speaking) around the hole that gives the illusion that the hole is moving back along the trajectory of the electron.  The hole then eventually collides into the bottom of the wire, whereupon an electron is pulled back into the nanowire; see Figure \ref{fig:electronFillsHole}.
\begin{figure}
\includegraphics[scale=.45]{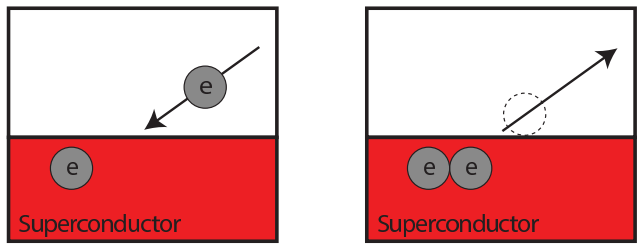}	
\caption{Ideally speaking, an electron reflects off of every side of the nanowire until it reaches the bottom of the wire.  It is there where the electron is pulled from the nanowire, forming a Cooper pair and a hole is left in its place.}
\label{fig:electronPairing}
\end{figure}
\begin{figure}
\includegraphics[scale=.45]{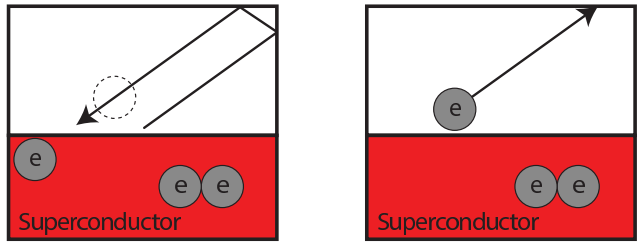}	
\caption{The hole returns to the bottom of the nanowire. At such a time, an electron is pulled into the nanowire and the process repeats.}
\label{fig:electronFillsHole}
\end{figure}

Seeking to reigorously define (mathematical) Andreev reflection, in our idealized setting, we will assign a parity to a pointmass.  The pointmass begins with positive parity.  Upon intersecting with the Andreev subset of an Andreev billiard (both of which to be formally defined in \S\ref{sec:PhysicallyInspiredBilliard}), the parity reverses and becomes negative, much like how the departure of an electron from a nanowire results in a hole.  Subsequent collisions in the Andreev subset cause the parity to change from negative to positive and from positive to negative.  

The second goal is twofold: 1) to rigorously define an Andreev billiard and 2) rigorously define the continuous and discrete flows on an Andreev billiard.  

Once we have rigorously defined an Andreev billiard, we will show that the continuous and discrete flows (as defined) will preserve the essence of an Andreev billiard, namely, the continuous flow will negate the volume form $dx\wedge dy\wedge d\theta$ when intersecting with the Andreev subset and the discrete flow will preserve the measure $\cos\phi drd\phi$.

The third goal will be to examine the properties of the continuous/discrete flow.  That is, we will show that under certain conditions, for almost every initial direction, the flow on almost any polygonal Andreev billiard will be closed.

Finally, taking motivation from experimental investigations of physical Andreev billiards, we will investigate the effect of a fractal perturbation a toy model of a rectangular Andreev billiard.  

\vspace{2 mm}

The paper is organized as follows.  In \S\ref{sec:PhysicallyInspiredBilliard}, the physics of an Andreev billiard is described and rigorous definitions of mathematical Andreev reflection and polygonal Andreev billiard are given.  The continuous flow $\andflow{t}$ on a polygonal Andreev billiard is described in \S\ref{sec:ctsFlow} and it is there that one sees that the absolute value of the volume element on the phase space of the billiard is preserved. In fact, the parity of the billiard ball of physical Andreev billiard is reflected in the fact that the volume element changes sign under the flow $\andflow{t}$.  In \S\ref{sec:discreteFlow}, the discrete flow is described and is shown to preserve the measure $\cos\phi d\phi dr$ on a section of the phase space.  Finally, in \S\ref{sec:perturbationsInTheAndreevSideOfOmegaA}, an investigation of the effects of a fractal perturbation of the toy model of a physical Andreev billiard is made.  The reason for this investigation is to present theoretic evidence in support of the argument that certain rough boundaries of physical Andreev billiards can be first studied as fractally perturbed boundaries.  A discussion of the results is then given and directions for future research are discussed.

\section{A physically inspired billiard}
\label{sec:PhysicallyInspiredBilliard}
We begin with a definition, with the end goal being a rigorous definition for Andreev reflection (given in the form of an equivalence relation $\sim_\pm$) and an Andreev billiard.

\begin{definition}[Retro-reflection]
Let $\Omega(B)$ be a mathematical billiard with boundary $B$.  Suppose $F$ is the billiard map acting on $(B\times S^1)/\sim$.  If, for some $n>0$, $(x^n,\theta^n) := F^n(x^0,\theta^0)$, $\theta^n = \theta^{n-1}+\pi$, then we say the orbit $\{F^k(x^0,\theta^0)\}_{k=0}^\infty$ has experienced retro-reflection at $x^n$.  
\end{definition}

In order to build towards our end goal of this section, we give a heuristic definition of a mathematical Andreev billiard inspired by the toy model of a physical Andreev billiard. To such end, suppose now that the phase space for a billiard is given by 
\begin{align}
	(\Omega(B)\times S^1)/\sim&\times \{-1,1\},
	\label{eqn:heuristicAndreevBilliard}
\end{align}
\noindent  giving us effectively two copies of the phase space.  As illustrated in Figure \ref{fig:PhysicalAndreevReflection}, in  a (mathematical) Andreev billiard, when the pointmass retro-reflects in a side of $\Omega(B)$, the parity of the pointmass changes.  Using our heuristic description given in (\ref{eqn:heuristicAndreevBilliard}), we model this by requiring that the pointmass has continued onto the other copy of the billiard, as shown in Figure \ref{fig:wrappingUpTables}.

Heuristically speaking, a pointmass has experienced Andreev reflection if the parity of the pointmass has changed upon retro-reflecting.  We define this formally as follows. 

Define $\Omega(\bplus)$  (or $\Omega(\bminus)$) to be the billiard with piecewise smooth boundary $\bplus$ (resp., $\bminus$).  $\Omega(\bplus)$ and $\Omega(\bminus)$ are billiards, whereupon the classical billiard dynamics are defined (excluding corners and cusps).  Let $\aplus$ and $\aminus$ be nonempty subsets of $\bplus$ and $\bminus$, respectively, containing no singularities of the billiard flow $\Phi_+$ defined on $(\Omega(\bplus)\times S^1)/\sim$ and $\Phi_-$ defined on $(\Omega(\bminus)\times S^1)/\sim$, where $\sim$ is the standard relation identifying inward and outward pointing vectors on the boundary of a billiard.  That is, at each point $\mathbf{x}=(x_1,x_2)$ of the boundary of a billiard, $(\mathbf{x},\theta)$ is identified with $(\mathbf{x},\gamma)$, where $\gamma = r(\theta)$, $r(\theta)$ being the reflection of $\theta$ through the tangent line through  $\mathbf{x}$.    Further assume that $\bminus = \rho(\bplus)$, where $\rho$ is the reflection through a line parallel to $\aplus$.  Consequently, it is the case that $\aminus = \rho(\aplus)$.  Without loss of generality, we may assume $\aplus$ and $\aminus$ are vertical segments in their respective billiards.\footnote{Recall that we are restricting our attention to sets in the plane with polygonal boundary.}

	First note that $X^\circ$ denotes the interior of a set $X$. Define $\sim_\pm$ on $(\Omega(\bplus)\cup\Omega(\bminus))\times S^1$ as follows: $(\mathbf{x},\theta)\sim_\pm (\mathbf{y},\gamma)$ if and only if 
	
	\begin{itemize}
	\item $\mathbf{x},\mathbf{y}\in \Omega(\bplus)^\circ\cup\Omega(\bminus)^\circ$ and $(\mathbf{x},\theta)= (\mathbf{y},\gamma)$ 

or
		\item $\mathbf{x},\mathbf{y}\in(\bplus\cup\bminus)\setminus  (\aplus\cup\aminus)$ and $\mathbf{x}=\mathbf{y}$ and 
		\begin{itemize}
			\item $\theta = \gamma$ or
			\item $\theta = r(\gamma)$
		\end{itemize}

	or
	\item  $\mathbf{x},\mathbf{y}\in \aplus\cup \aminus$ and 
		\begin{itemize}
			\item $\mathbf{x}=\mathbf{y}$ and $\theta=\gamma$ or
			\item $\mathbf{x} = \rho(\mathbf{y})$ and $\theta = r(\gamma) +\pi$
		\end{itemize}
	\end{itemize}
	
	The relation $\sim_\pm$ is an equivalence relation on $\Omega(\bplus)\cup\Omega(\bminus)\times S^1$.  As in the classical case, we choose the representative element of an equivalence class to be the ordered pair $(\mathbf{x},\theta)$, where $\theta$ is an inward pointing direction based at $\mathbf{x}$.  We now formally define \emph{Andreev billiard}.

\begin{definition}[Polygonal Andreev billiard]
Let $\Omega(\bplus)$ be a polygonal billiard, $\aplus \subseteq \bplus$ be a nonempty, connected $C^1$ subset of $\bplus$, where $\aplus$ is a vertical segment of $\bplus$ and $\Omega(\bminus) := \rho(\Omega(\bplus))$, with $\rho(\cdot)$ being the reflection of a set through a line parallel to $\aplus$ and not intersecting the interior of $\Omega(\bplus)$.  Define $\aminus :=\rho(\aplus)$ and let $\Omega(\bplus)$ and $\Omega(\bminus)$ be glued along $\aplus$ and $\aminus$ according to $\sim_\pm$.  Then, $\Omega(\bplus)\cup \Omega(\bminus)$ is called a \emph{polygonal Andreev billiard} and the phase space for the billiard is $(\Omega(\bplus)\cup\Omega(\bminus))\times S^1/\sim_\pm$. 	
\end{definition}

\begin{notation}
So as to simplify notation and improve the readability of the remainder of the article, we will denote an Andreev billiard $\Omega(\bplus)\cup\Omega(\bminus)$ by $\andbill$.  That is, $(\Omega(\bplus)\cup\Omega(\bminus))\times S^1/\sim_\pm = \andbill\times S^1/\sim_\pm$.  Similarly, $B^\pm := \bplus \cup \bminus$.
\end{notation}

\begin{definition}[Andreev subset]
Let $\andbill$ be a polygonal Andreev billiard.  The subset $\aplus\cup\aminus$ is the \emph{Andreev subset} of the polygonal Andreev billiard.  For the sake of readability, we denote the Andreev subset by  $A^\pm$.

\end{definition}

Really, the Andreev subset is two subsets of the two copies of a polygonal billiard properly identified.  Upon identifying point-angle pairs, one sees that the $\aplus$ and $\aminus$ are effectively the same in the phase space $\andbill\times S^1/\sim_\pm$.

Figure \ref{fig:AndreevBilliardTable} gives an example of an orbit of an Andreev billiard; the same orbit is illustrated as in Figure \ref{fig:gluedTables}, but, as we will see, is given by a well-defined flow on $\andbill\times S^1/\sim_\pm$.

\begin{figure}
\centering
\includegraphics[scale=.8]{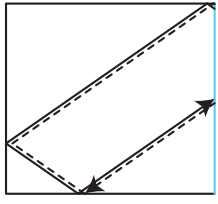}
\caption{The blue side (or the lighter right hand vertical side for those not seeing the color in the figure) of the billiard is the Andreev subset of the Andreev billiard.  Any trajectory retro-reflects in the blue side.  Upon retro-reflecting, the parity of the billiard ball changes from positive to negative (or from negative to positive), mimicking the physical situation where an electron disappears and a hole 'moves' back along the trajectory.}
\label{fig:PhysicalAndreevReflection}
\end{figure}

\begin{figure}
	
\includegraphics{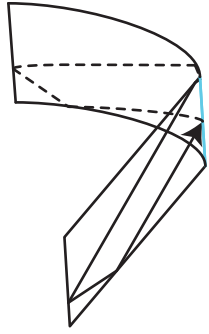}
\caption{One can see that the retro-reflection with parity illustrated in Figure \ref{fig:PhysicalAndreevReflection} is actually a flow on two copies of a billiard appropriately glued together with a nonstandard flow across the side on which the two billiards are glued.  The glued copies of the billiard are then illustrated in the plane in Figure \ref{fig:gluedTables}.} 
\label{fig:wrappingUpTables}
\end{figure}

\begin{figure}
	\includegraphics[scale=.8]{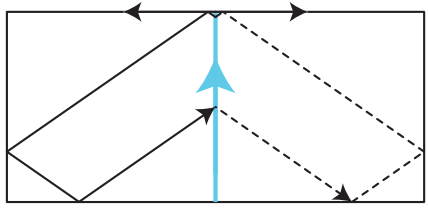}
	\caption{Here we see two copies of a billiard glued together along an edge.  The orientation of the parameterization of the boundary  of the billiard on the right is opposite that of the billiard on the left.  Reflecting the left billiard along the vertical edge common to both results in exactly the image shown in Figure \ref{fig:PhysicalAndreevReflection}.}
	\label{fig:gluedTables}
\end{figure}

\section{Volume parity under the flow $\andflow{t}$}
\label{sec:ctsFlow}
We now define the flow $\andflow{t}$ on an Andreev billiard and show that the flow preserves the absolute value of the volume of a connected set in $\andbill\times S^1/\sim_\pm$.  The map
	\begin{align}
\notag	\andflow{t}:(\andbill\times S^1)/\sim_\pm &\to (\andbill\times S^1)/\sim_\pm
	\end{align}
	\noindent is well defined for all $t\in \mathbb{R}$, so long as $\andflow{t}((x,y),\theta)\neq ((x^t,y^t),\theta^t)$, where $(x^t,y^t)$ is a corner or cusp of $\Omega(\bplus)\cup\Omega(\bminus)$.  Additionally, when defined, 
\begin{align}
\notag \andflow{t}((x,y),\theta) &= ((x^t,y^t),\theta^t),
\end{align}

\noindent where $\theta^t$ is inward pointing at $(x^t,y^t)$; see Figure \ref{fig:AndreevBilliardTable} for an example of a billiard orbit on $\andbill$.

\begin{figure}
	\begin{overpic}{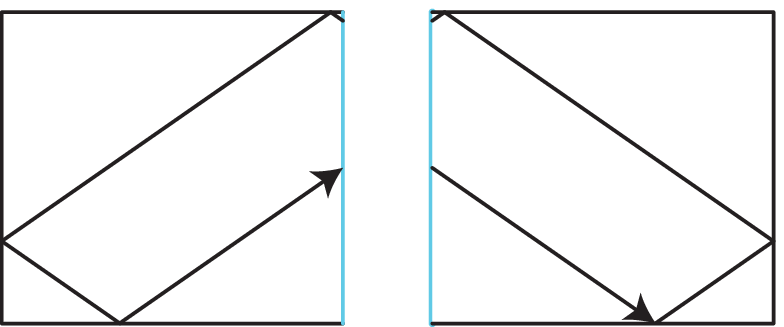}
	\put (17,-5) {$\Omega(\bplus)$}
	\put (72,-5) {$\Omega(\bminus)$}
	\put (44.5,30) {$\aplus$}
	\put (48.5, 6) {$\aminus$}
\end{overpic}
\caption{An orbit of an Andreev billiard $\andbill$.}
\label{fig:AndreevBilliardTable}
\end{figure}

\begin{theorem}
The volume element $dx\wedge dy\wedge d\theta$ on $(\andbill\times S^1)/\sim_\pm$ is not preserved under $\andflow{t}$.  Furthermore, for $\theta,\theta'\in S^1$, $(x,y)\in \Omega(\bplus)^\circ$, $(x',y')\in \Omega(\bminus)^\circ$ and $t>0$ such that $\andflow{t}((x,y),\theta) = ((x',y'),\theta')$, it is the case that
\[
dx'\wedge dy'\wedge d\theta' = - dx\wedge dy\wedge d\theta.
\]
\end{theorem}

\begin{proof}
The obvious concern is how $dx\wedge dy \wedge d\theta$ is affected by a collision in the boundary $B^\pm$.  Since the map $\andflow{t}$ preserves the volume element after each collision in a segment not contained in $\aplus\cup\aminus$, we focus our proof on trajectories that would intersect the Andreev subset $\aplus\cup \aminus$.  Consider $x,y,x',y',\bar{x},\bar{y}$ as shown in Figure \ref{fig:AndreevReflection}.  The angle $\gamma$ is the angle the tangent $r$ makes with the positive horizontal axis; $\psi$ is the angle between the dotted trajectory and $r$; $\theta$ is the initial angle and $\theta'$ is the angle at which the billiard ball continues on after colliding at $(\bar{x},\bar{y})$.  Then,
\begin{align}
\notag	x'  & = \bar{x}+s'\cos \theta\\
\notag	dx' &= d\bar{x} + d(s'\cos\theta)\\
\notag			&= dr \cos \gamma + ds'\cos\theta - s'\sin\theta d\theta, \\
\notag &\\
\notag	y' &= \bar{y} - s'\sin\theta\\
\notag    dy'&= d\bar{y} - d(s'\sin\theta)\\
\notag    	& = dr\sin \gamma - ds'\sin\theta - s'\cos\theta d\theta,\\
\tag*{and} &\\
\notag	d\theta ' &= d(\theta+2\psi +\pi)\\
\notag					& = d\theta + 2d\psi \\
\notag					& = -d\psi + 2d\psi \\
\notag					& = d\psi.
\end{align}

\noindent Consequently,
\begin{align}
\notag	dx'\wedge dy'\wedge d\theta' &= dx'\wedge dy'\wedge d\psi \\
\notag											  & = (-\cos\gamma\sin\theta dr\wedge ds' - \cos\theta  \sin\gamma dr\wedge ds)\wedge d\psi\\
\notag									  &  = -\sin (\theta+\gamma)dr\wedge ds'  \wedge d\psi\\
\notag										  & = \sin (\theta +\gamma) dr\wedge ds \wedge d\psi.
\end{align}

\noindent  We note that $\theta+\gamma = \pi - \psi$.  It then follows that

\begin{align}
\notag	dx'\wedge dy'\wedge d\theta' &= \sin(\pi - \psi)dr\wedge ds\wedge d\psi \\
\notag		&=\sin \psi dr\wedge ds\wedge d\psi.
\end{align}

 Since $dx\wedge dy\wedge d\theta = -\sin\psi dr \wedge ds\wedge d\psi$, we see that the volume element changes sign when a trajectory intersects $\aplus$ or $\aminus$. 
\end{proof}

\begin{figure}
	
\begin{overpic}{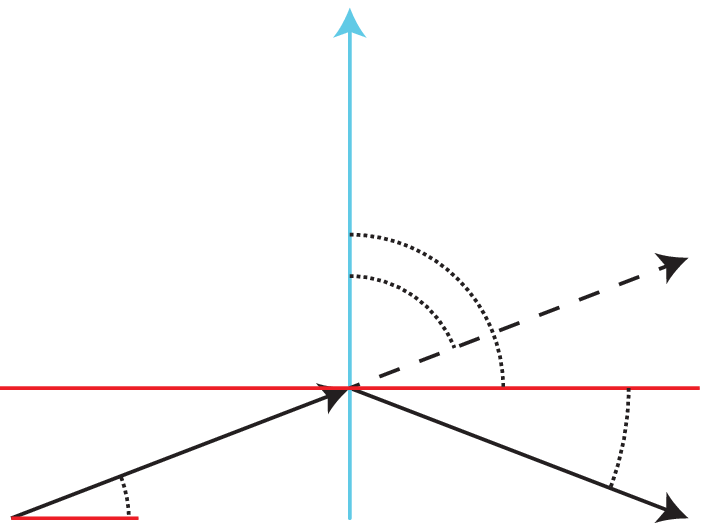}
 \put (-6,-3) {$(x,y)$}
 \put (38,22) {$(\bar{x},\bar{y})$}
 \put (43,40) {$r$}
 \put (98,-3) {$(x',y')$}
 \put (20,3) {$\theta$}
 \put (60,33) {$\psi$}
 \put (65,37) {$\gamma$}
 \put (91,10) {$2\pi-\theta'$}
\end{overpic}
\caption{A diagram of the billiard flow on $(\andbill\times S^1)/\sim_\pm$.}
\label{fig:AndreevReflection}
\end{figure}

The fact that the volume element changes sign captures the essence of the physical Andreev reflection in the toy model: the pointmass represents either an electron or a hole, and what the pointmass represents changes each time with a `collision' in the Andreev subset.  More precisely, in the physical toy model, there is a particular parity: hole or electron.  In our rigorously defined setting, we see that the volume of an open set $X$ contained in the interior of $\Omega(\bplus)$ will be negated as (and if) $\Phi^t$ transports $X$ into $\Omega(\bminus)$.  That is, if in a fixed direction $\theta$, $\Phi^t(X)\subseteq \Omega(\bminus)^\circ$, then the volume of $\Phi^t(X)$ in $\Omega(\bminus)$ will be opposite that of the volume of $X$, much like how an electron colliding with the Andreev subset (the side lying upon a superconducting medium), is replaced by a hole and such a hole is then subsequently replaced by another electron.  Again, one observes a physical phenomenon and attempts to construct a rigorous mathematical setting in which the toy model of such a physical setting is rigorously described.  The volume parity of a set rigorously captures the electron-hole parity observed in the toy model.
  \subsection{A closed flow}
  \begin{definition}[Rational polygonal Andreev billiard]
An polygonal Andreev billiard $\andbill$ is a \emph{rational polygonal Andreev billiard} if $\bplus$ is a rational polygon, meaning that the interior angles of $\bplus$ are rational multiples of $\pi$.	
\end{definition}

In a fixed direction $\theta$, one may consider the subspace 
\[
M^\pm_\theta:=(\andbill\times \{\theta,\theta_1,\theta_2,...,\theta_n\})/\sim_\pm\subseteq (\andbill\times S^1)/\sim_\pm,
\]
 where $\theta_1,...,\theta_n$ are all the possible angles of reflection experienced by the billiard ball when colliding in the sides of $\andbill$.  When one restricts the flow $\andflow{t}t$ to $M^\pm_\theta$, the following result holds.

\begin{theorem}
	For  any rational polygonal Andreev billiard $\andbill$ and almost any direction $\theta$ (with respect to the Lebesgue measure), the flow $\andflow{t}|_{M^\pm_\theta}$ is closed.
	 \end{theorem}

\begin{proof}
Kerckhoff, Masur and Smillie proved that for any rational billiard $\Omega(B)$ and almost any direction $\theta$, the billiard flow $\Phi_\theta^t$ on $M_\theta$ is uniquely ergodic; see \cite{KeMasSm1, KeMasSm2}.\footnote{In the classical case, $M_\theta :=(\Omega(B)\times \{\theta,\theta_1,\theta_2,...,\theta_n\})/\sim_\pm\subseteq (\andbill\times S^1)/\sim$, where $\theta_i$, $i\leq n$, are the angles of reflection made by the pointmass when colliding in the boundary $B$.}  Consider the rational polygonal Andreev billiard table $\andbill$.  The flow $\andflow{t}$ restricted to $(\Omega(\bplus) \setminus \aplus\times S^1)/\sim_\pm$ is identical to the flow $\Phi_+^t$ restricted to $(\Omega(\bplus)\times S^1)/\sim$. Therefore, for almost every direction $\theta$ and for every $x\in \Omega(\bplus$, there exists $t_0$ such that $\Phi_+^{t_0}(x,\theta) \in (\aplus\times S^1)/\sim$.

Let $(y,\gamma) = \Phi^{t_0}_+(x,\theta)$.  Then, $r(\gamma)+\pi$ based at $\rho(y)$ is inward pointing. Given the nature of the reflection at the Andreev subset, $\Phi_-^{t_0}(\rho(y),r(\gamma)+\pi) = (\rho(x),\theta+\pi)$.  Additionally, there exists $t_1>0$ such that $\Phi^{-t_1}_+(x,\theta)\in (\aplus\times S^1)/\sim$.  Consequently, $\Phi^{t_1}_-(\rho(x),\theta+\pi)\in (\aminus\times S^1)/\sim$.  If $z$ is the base-point of $\Phi^{-t_1}_+(x,\theta)$, then $\rho(z)$ is the base-point of $\Phi^{t_1}(\rho(x),r(\theta)+\pi)$.  We conclude that $\andflow{2t_0+t_1}(x,\theta) = (x,\theta)$ for every $x\in \Omega(\bplus)$.  By a similar argument, the flow $\andflow{t}(x,\theta)$ is closed for every $x\in \Omega(\bminus)$.  Therefore, the rational polygonal Andreev billiard flow is closed for almost every direction $\theta$.

\end{proof}

\section{Preservation of the measure $\cos\phi drd\phi$}
\label{sec:discreteFlow}

The collision map $F_\pm: (B^\pm\times S^1)/\sim_\pm \to (B^\pm\times S^1)/\sim_\pm$ (when such a map is defined), preserves the measure $\cos\phi dr d\phi$.  The derivative of $F_\pm$ is entirely dependent on the curvature of the segments comprising $\bplus$ and $\bminus$ and the initial direction $\theta$.  We then denote the parameterization of the side on which $(\bar{x}',\bar{y}')$ lies by $r'$; see Figure \ref{fig:differentiable}.

\begin{figure}
	\begin{overpic}[scale=.9]{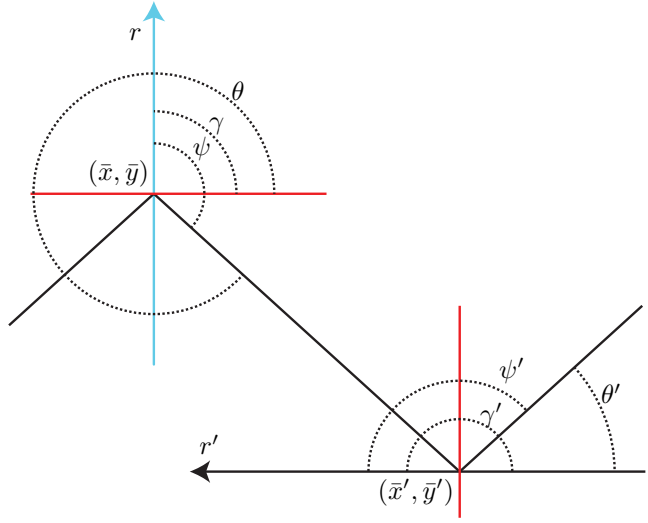}
	\put(19,75){$r$}	
	\put(29,57){$\psi$}
	\put(31.5,61){$\gamma$}
	\put(35,65.5){$\theta$}
	\put(74,15){$\gamma'$}
	\put(77,22){$\psi'$}
	\put(93,18){$\theta'$}
	\put(12.5,53){$(\bar{x},\bar{y})$}
	\put(58,3){$(\bar{x}',\bar{y}')$}
	\put (30,10) {$r'$}
	\end{overpic}
	\caption{The map $F_\pm$ is piecewise differentiable.}
	\label{fig:differentiable}
\end{figure}

Figure \ref{fig:differentiable} demonstrates how a pointmass is reflected in the Andreev subset of $\andbill$.\footnote{More precisely, $\aplus$ and $\aminus$ have been identified in the phase space.}  The subset $\aplus$ (or $\aminus$) has a parametrization given by $r$. We  note that $\theta$ is the direction of the trajectory, $\theta'$ the angle at which the pointmass is reflected off of $r'$.  The angles $\theta$ and $\theta'$ are measured relative to positive horizontal axis, as shown.  The angles $\gamma$ and $\gamma'$ measure the angle between the  positive horizontal axis and the sides $r$ and $r'$, respectively.  The angles $\psi$ and $\psi'$ measure the angle between $r$ and $r'$, respectively.

\begin{lemma}
	Let $\theta$, $\gamma$, $\gamma'$, $\psi$ and $\psi'$ be as described in the preceding paragraph and illustrated in Figure \ref{fig:differentiable}.  Then, $\theta \bmod{\pi}  = (\gamma'+\psi')\bmod{\pi} = (\gamma -\psi)\bmod{\pi}$.
\end{lemma}

\begin{theorem}
Suppose $\bplus$ is a polygon.  The map $F_\pm$ preserves the measure $\cos \phi dr d\phi$ on the space $(B^\pm\times S^1)/\sim_\pm$.  
\end{theorem}

\begin{proof}
	It is known that the billiard map $F_\pm$ is differentiable when collisions occur in $\bplus\setminus \aplus$ or $\bminus\setminus\aminus$ and away from corners or other singularities of the flow.  Additionally, the billiard map $F_\pm$ preserves the measure $\cos\phi drd\phi$ when collisions occur on $\bplus\setminus \aplus$ or $\bminus\setminus\aminus$, since $F_\pm$ is identical to $F$ on $\bplus\setminus\aplus$ and $\bminus\setminus\aminus$.
	
	What follows is focused on showing that $F_\pm$ is differentiable at $(\bar{x},\bar{y})\in \aplus$.\footnote{Recall that we are assuming that $\bplus$ and $\bminus$ are polygons.}  The boundary $\bplus$ of $\Omega(\bplus)$ can be parameterized.  Suppose $\aplus$ is parameterized by $r$. The orientation of the parameterization of $\bminus$ is opposite that of $\bplus$.  
	
	First, we note that the length of the segment between $(\bar{x},\bar{y})$ and $(\bar{x}',\bar{y}')$ is denoted by $\tau$.  Then,
	\begin{align}
\notag		\bar{x}' - \bar{x}  = \tau\cos\theta\quad&\quad \bar{y}' - \bar{y}  = \tau\sin\theta
	\end{align}
	
Without loss of generality, we assume that $\pi<\theta<2\pi$ so that $\theta = 2\pi +\gamma-\psi = \gamma'+\psi'$.  Then,
\begin{align}
\notag	d\theta = d\gamma &-d\psi  = d\gamma'+d\psi'.
\end{align}
	
Since $d\gamma = d\gamma' =0$, we see that $-d\psi = d\psi'$.  Continuing, we see that 
	\begin{align}
\notag		d(\bar{x}' - \bar{x}) &= d(\tau\cos\theta)\\
\notag		\cos\gamma'dr' - \cos\gamma dr &= \cos\theta d\tau -\tau \sin\theta d\theta
	\end{align}
\noindent and
	\begin{align}
\notag		d(\bar{y}' - \bar{y}) &= d(\tau\sin\theta)\\
\notag		\sin\gamma'dr' - \sin\gamma dr &= \sin\theta d\tau +\tau \cos\theta d\theta
	\end{align}
As in the non-Andreev case (i.e., the case where $\Omega(B)$ is a polygonal billiard and every point $\mathbf{x}$ of $B$ at which a well-defined tangent exists, the classic law of reflection dictates how a pointmass continues on after colliding at $\mathbf{x}$), we compare coefficients of trigonometric expressions to obtain
\begin{align}
\notag	\sin \psi' dr' +\sin \psi dr & = \tau d\theta.
\end{align}

Letting $\psi =  \frac{\pi}{2} - \phi$ and $\psi' = \frac{\pi}{2}-\phi'$, we have that 
\begin{align}
\notag	\cos \phi' dr' +\cos \phi dr & = \tau d\theta.
\end{align}

\noindent Since $d\gamma=d\gamma' = 0$, we have that 
\begin{align}
\notag	-\cos\phi' dr' &= \cos\phi dr+\tau d\phi 
\end{align}

\noindent and
\begin{align}
\notag-\cos\phi'd\phi' &= \cos\phi'd\phi 
\end{align}

 The derivative of $F_\pm$ at $(r,\phi)$ is then 
\begin{align}
\notag	D_{(r,\phi)}F_\pm &= \frac{-1}{\cos\phi'} \left[\begin{array}	{cc}
		\cos \phi & \tau \\ 0 &\cos\phi'\end{array}\right].
\end{align}

\noindent We then see that the determinant of $D_{(r,\phi)}F_\pm$ is $\cos\phi/\cos\phi'$ and
\begin{align}
\notag	\iint_{F_\pm(D)} \cos \phi' dr'd\phi' &= \iint_D \cos\phi drd\phi.
\end{align}
\end{proof}

\section{A fractal perturbation of an Andreev billiard}

\label{sec:perturbationsInTheAndreevSideOfOmegaA}
The following is motivated by  the toy model of the physical setting of the electron-hole dynamics in a metallic nanowire with three ``sides'' of the wire exhibiting the standard law of reflection and the other being the superconducting side, as  illustrated in Figure \ref{fig:PhysicalAndreevReflection}.  We give examples where particular perturbations may pose a problem for reliably modeling the dynamics of the electron-hole flow in a nanowire.  

Consider a toy model of a physical rectangular Andreev billiard $\Omega(R)$ with rationally commensurate side lengths (the ratio of the lengths of any two sides is a rational value) and the base of $\Omega(R)$ being the subset of $R$ at which retro-reflection with parity occurs; see Figures \ref{fig:electronPairing}--\ref{fig:PhysicalAndreevReflection} for a reminder of exactly how the toy model for physical Andreev reflection is behaving.

Consider the perturbation of $\Omega(R)$ given in Figure \ref{fig:rectangularNotchPerturbationNoOrbit}.   The effect of perturbing the wire by introducing this rectangular region is shown in Figure \ref{fig:rectangularNotchPerturbation}.  We see in Figure \ref{fig:rectangularNotchPerturbation} that the perturbation of the side has the effect of forcing the billiard ball back along a path in the same direction of its entry or along the direction as expected in the absence of the perturbation, depending on where the billiard ball began.  Indeed, for any direction, one can construct a rectangular perturbation where the choice of initial base point results in the perturbation affecting the billiard ball differently.

\begin{figure}
\begin{center}
\includegraphics[scale=.7]{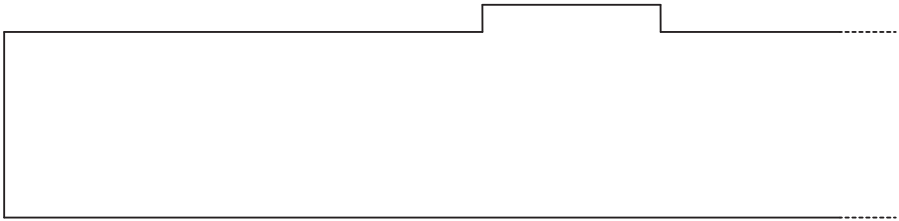}
\end{center}
\caption{A perturbation of $\Omega(R)$, where $\Omega(R)$ is a toy model for physical rectangular Andreev billiard.  There are naturally occurring defects in any nanowire.  We represent here an idealized model of a nanowire with a fairly unrealistic perturbation, but a reasonable representation of a perturbation.  Recall that a nanowire is much longer than it is tall ($1,000$ nanometers vs. $10$ nanometers).  Hence, any perturbation will be minuscule, relative to the overall length of the wire.}
\label{fig:rectangularNotchPerturbationNoOrbit}
\end{figure}

\begin{figure}
\begin{center}
\includegraphics[scale=.7]{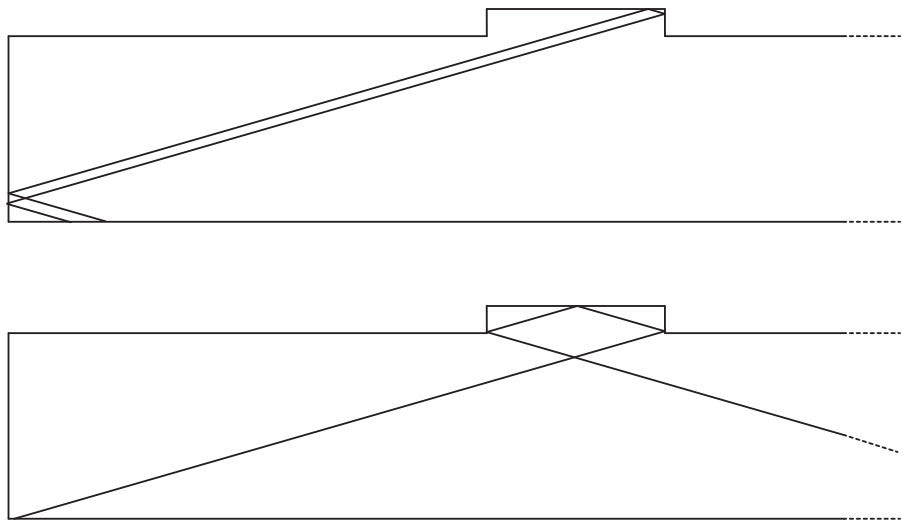}
\end{center}
\caption{The effects of the perturbation shown in Figure \ref{fig:rectangularNotchPerturbationNoOrbit} on two different orbits in the same direction, but with different initial starting points behave very differently.}
\label{fig:rectangularNotchPerturbation}
\end{figure}

In \cite{LapMilNie}, billiard dynamics on the so-called $T$-fractal billiard $\Omega(\mathscr{T})$ are studied; see Figure \ref{fig:tFractal}.\footnote{To the best of the author's knowledge, the $T$-fractal was first introduced in \cite{AcST}.} The $T$-fractal billiard is not a classical billiard; it is a fractal billiard as discussed in \cite{LapNie1,LapNie2,LapNie3,LapNie4}.  However, in particular rational directions (of which there are countably infinitely many), one can determine periodic orbits of $\Omega(\mathscr{T})$. Such orbits reach the top of the $T$-fractal billiard $\Omega(\mathscr{T})$, and, effectively retro-reflect, but without a change in parity.  Upon returning to the bottom of the fractal billiard, the ball continues accordingly; see, e.g.,  Figure \ref{fig:periodicOrbitTFractal}.  Specifically, Theorem 4.9 in \cite{LapMilNie} describes a countable family of directions and a dense set of initial basepoints from which one can construct a periodic degenerate orbit of the $T$-fractal billiard.  We restate the theorem here for the reader's easy reference.

\begin{theorem}[\cite{LapMilNie}]
\label{thm:LapMilNie}
Let $p>1$ be an odd, positive integer and $x_{0}^0 \neq m2^{-l}$, $l,m$ being positive integers.  Suppose $\{\mathscr{O}_n(x^{0}_n,\theta^0)\}_{n=i}$ is a sequence of compatible periodic orbits with $\tan{\theta^0} = \frac{1}{p}$.  Then, the Hausdorff-Gromov limit $\mathscr{O}(x^0,\theta^0)$ of the sequence of compatible periodic orbits is a periodic orbit of $\Omega(T)$.	
\end{theorem}

\begin{figure}
\begin{center}
\includegraphics[scale=.7]{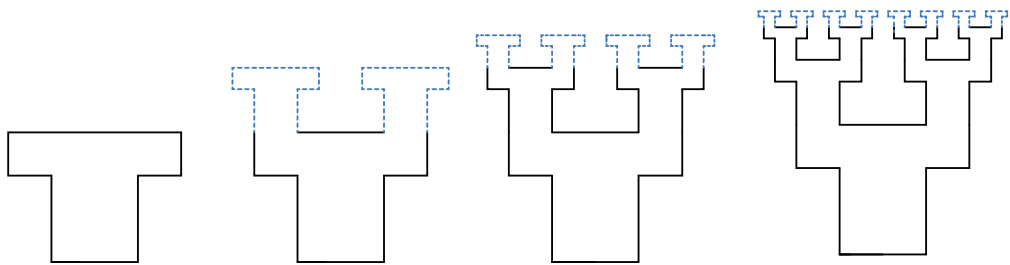}
\end{center}
\caption{The T-fractal billiard.}
\label{fig:tFractal}
\end{figure}

\begin{figure}
\begin{center}
\includegraphics[scale=.6]{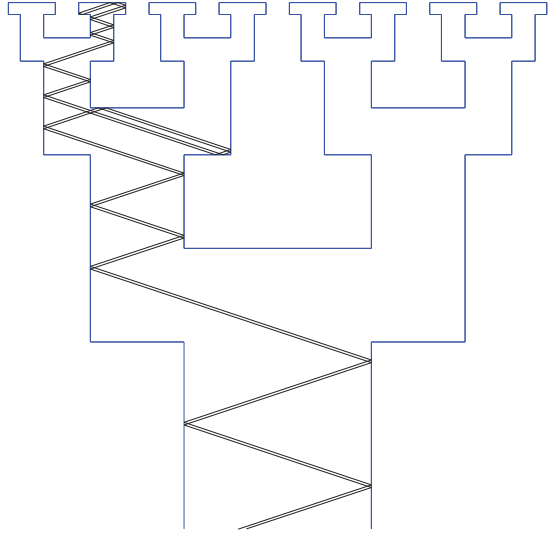}
\end{center}
\caption{A periodic orbit of the $T$-fractal billiard continuing upon returning to the bottom of $\Omega(\mathscr{T})$.}
\label{fig:periodicOrbitTFractal}
\end{figure}

\begin{figure}
\begin{center}
\includegraphics[scale=.7]{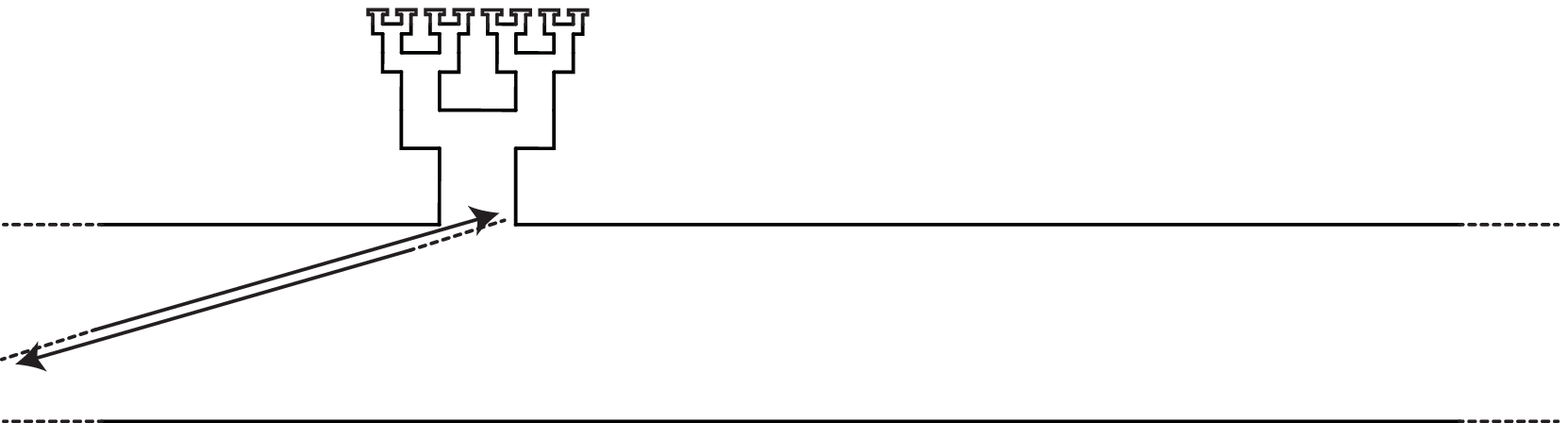}
\end{center}
\caption{A $T$-fractal perturbation of a nanowire.}
\label{fig:tFractalNotchPerturbationWithorbit}
\end{figure}

Suppose we consider a perturbation of a nanowire in the shape of the $T$-fractal billiard $\Omega(\mathscr{T})$.  Supposing we consider an electron-hole (or a billiard ball with parity) entering into such a perturbation at a point and with an initial direction described in the hypotheses of Theorem  \ref{thm:LapMilNie}, then such a perturbation will cause the electron-hole to return to the rectangular region of the nanowire antiparallel to the direction in which it entered, this being discussed in the proof of Theorem 4.9 in \cite{LapMilNie}.  This sort of perturbation causes the orbit of the nanowire to behave as if the electron-hole had intersected a vertical side at or near the corner of the nanowire; see Figure \ref{fig:tFractalNotchPerturbationWithorbit} and Figure \ref{fig:effectivelyAnEdge}.

\begin{figure}
\begin{center}
\includegraphics[scale=2]{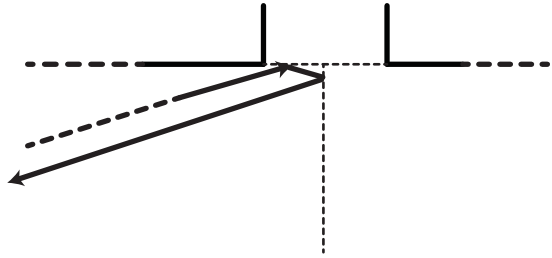}
\end{center}
\caption{For particular directions, an orbit will enter the $T$-fractal perturbation of the nanowire and exit in a direction anti-parallel to the direction of entry.  The dashed vertical line is to indicate where the nanowire is effectively severed and how the orbit is then behaving as if the billiard ball had intersected the edge near the corner.}
\label{fig:effectivelyAnEdge}
\end{figure}

To be clear, never\footnote{in any reasonable amount of time} will anyone be able to perturb a nanowire in such a way that the $T$-fractal appears, simply because the $T$-fractal is only a set that exists in one's mind.  For that matter, and for all intents and purposes, any approximation of the $T$-fractal is not a reasonable or realistic perturbation of a rectangular nanowire.  Nevertheless, whether it be in a physical setting or an abstract setting, one may find the existence of a fractal perturbation interesting in its own right or inspiring some other train of thought not yet foreseen by the author.  The argument made in this paper is that a rectangular billiard is a reasonable approximation to a nanowire and that the evidence of a particular perturbation mimicking the severing of a nanowire may highlight a new perturbation worth examining on the physical side of the subject, if perhaps only in computer simulations.

\section{Concluding remarks}

From \S\ref{sec:ctsFlow}, it is clear that the flow on a rational polygonal Andreev billiard in almost any direction will be closed.  The volume parity discussed in \S\ref{sec:ctsFlow} is analogous to the electron-hole parity modeled in Figures \ref{fig:electronPairing} and \ref{fig:electronFillsHole}.  As expected, however, the discrete flow $F_\pm^n$ preserves the measure $\cos\phi d\phi dr$ on the space $(\andbill\times S^1)/\sim_\pm$.  

In returning to the toy model, a discussion of perturbations of a nanowire was given in \S\ref{sec:perturbationsInTheAndreevSideOfOmegaA}.  It was there that the $T$-fractal was shown to effectively sever a nanowire by forcing all trajectories in a particular fixed direction to retroreflect back into the rectangular region.  Such a theoretical investigation sought to shed light on the idea that fractal perturbations of polygonal billiard tables (esp., polygonal Andreev billiard tables) are meaningful.

Much of this article has been on the topic of polygonal Andreev billiards.  One can reasonably expect the same rigorously defined Andreev reflection to carry over to a more general setting where the boundary of an Andreev billiard is not the union of two polygonal billiards but two copies of a piecewise smooth curve properly identified.  With suitable adjustments to the construction of an Andreev billiard, one can then discuss a more general Andreev billiard, analyze the flow and determine whether or not in almost any direction the flow would remain closed.  It is unlikely that the flow in almost any direction on an arbitrary Andreev billiard will be closed since there are known examples of billiard tables whereupon the flow is not equidistributed nor minimal; see \cite{Sm}.  Additionally, it is conceivable that there exists a billiard such that for some fixed direction $\theta$, every orbit is not closed, but also not dense, meaning the orbit gets trapped in some proper subset of the table.  

Future works will focus on generalizing the notion of an Andreev billiard and, simultaneously, investigating the effects of fractal perturbations of polygonal billiards.

\subsubsection*{Acknowledgments}

The author would like to thank Terry A. Loring and John Huerta for thoughtful comments on numerous drafts of this paper and Terry A. Loring for suggesting the original project.

\end{document}